\documentclass[a4paper,12 pt,twoside]{article}

\usepackage{amsmath,amssymb,amsthm}
\usepackage[T1]{fontenc}

\newtheorem{theorem}{Theorem}[section]
\newtheorem{corollary}[theorem]{Corollary}

\newtheorem{proposition}[theorem]{Proposition}

\theoremstyle{definition}
\newtheorem{definition}[theorem]{Definition}
\theoremstyle{definition}
\newtheorem{remark}[theorem]{Remark}
\theoremstyle{definition}
\newtheorem{example}[theorem]{Example}

\numberwithin{equation}{section}

\title{On $(h,k,\mu,\nu)$-trichotomy of evolution operators in Banach spaces}
\date{}
\begin{document}


\maketitle

\author{ Mihail Megan, Traian Ceau\c su, Violeta Crai}

\begin{abstract}
The paper considers some concepts of trichotomy with different growth rates for evolution operators in Banach spaces. Connections between these concepts and characterizations in terms of Lyapunov- type norms are given.
\end{abstract}
\section{Introduction}
In the qualitative theory of evolution equations, exponential dichotomy, essentially introduced by O. Perron in \cite{peron} is one of the most important asymptotic properties and in last years it was treated from various perspective.

For some of the most relevant early contributions in this area we refer to the books of J.L. Massera and J.J. Schaffer \cite{masera}, Ju. L. Dalecki  and M.G. Krein \cite{daletchi} and W.A. Coppel \cite{copel}. We also refer to the book of C. Chichone and Yu. Latushkin \cite{chicone}.

In some situations, particularly in the nonautonomous setting, the concept of uniform exponential dichotomy is too restrictive and it is important to consider more general behaviors. Two different perspectives can be identify for to generalize the concept of uniform exponential dichotomy: on one hand one can define dichotomies that depend on the initial time (and therefore are nonuniform) and on the other hand one can consider growth rates that are not necessarily exponential.

The first approach leads to concepts of nonuniform exponential dichotomies and can be found in the works of L. Barreira and C. Valls \cite{ba2} and in a different form in the works of P. Preda and M. Megan \cite{preda-megan} and M. Megan, L. Sasu and B. Sasu \cite{megan-sasu}.

The second approach is present in the works of L. Barreira and C. Valls \cite{ba3}, A.J.G. Bento and C.M. Silva \cite{bento} and M. Megan \cite{megan}.

A more general dichotomy concept is introduced by M. Pinto in \cite{pinto2} called $(h,k)$-dichotomy, where $h$ and $k$ are growth rates. The concept of $(h,k)-$ dichotomy has a great generality and it permits the construction of similar notions for systems with dichotomic behaviour which are not described by the classical theory of J.L. Massera \cite{masera}.

As a natural generalization of exponential dichotomy (see \cite{ba3}, \cite{vio}, \cite{elaydi1}, \cite{saker}, \cite{sasu} and the references therein), exponential trichotomy is one of the most complex asymptotic properties of dynamical systems arising from the central manifold theory (see \cite{carr}). In the study of the trichotomy the main idea is to obtain a decomposition of the space at every moment into three closed subspaces: the stable subspace, the unstable subspace and the central manifold.

Two concepts of trichotomy have been introduced: the first by R.J. Sacker and G.L. Sell \cite{saker} (called (S,S)-trichotomy) and the second by S. Elaydi and O. Hayek \cite{elaydi1} (called (E,H)-trichotomy).

The existence of exponential trichotomies is a strong requirement and hence it is of considerable interest to look for more general types of trichotomic behaviors.

In previous studies of uniform and nonuniform trichotomies, the growth rates are always assumed to be the same type functions. However, the nonuniformly hyperbolic dynamical systems vary greatly in forms and none of the nonuniform trichotomy can well characterize all the nonuniformly dynamics. Thus it is necessary and reasonable to look for more general types of nonuniform trichotomies.

The present paper considers the general concept of nonuniform $(h,k,\mu,\nu)-$ trichotomy, which not only incorporates the existing notions of uniform or nonuniform trichotomy as special cases, but also allows the different growth rates in the stable subspace, unstable subspace and the central manifold.

We give characterizations of nonuniform $(h,k,\mu,\nu)-$ trichotomy using families of norms equivalent with the initial norm of the states space. Thus we obtain a characterization of the nonuniform $(h,k,\mu,\nu)-$trichotomy in terms of a certain type of uniform $(h,k,\mu,\nu)-$trichotomy.

As an original reference for considering families of norms in the nonuniform theory we mention Ya. B. Pesin's works \cite{pesin} and \cite{pesin1}. Our characterizations using families of norms are inspired by the work of L. Barreira and C. Valls \cite{ba3} where characterizations of nonuniform exponential trichotomy in terms of Lyapunov functions are given.
\section{Preliminaries}
Let $X$ be a Banach space and $\mathcal{B}(X)$ the Banach algebra of all linear and bounded operators on $X$. The norms on $X$ and on $\mathcal{B}(X)$ will be denoted by $\|\cdot\|$. The identity operator on $X$ is denoted by $I$. We also denote by $\Delta=\{(t,s)\in\mathbb{R}_+^2:t\geq s\geq 0\}$.

We recall that
an application $U:\Delta\to\mathcal{B}(X)$ is called \textit{evolution operator} on $X$ if
\begin{itemize}
	\item[$(e_1)$]$U(t,t)=I$, for every $t\geq 0$
	\item[] and
	\item[$(e_2)$]$U(t,t_0)=U(t,s)U(s,t_0)$, for all $(t,s),(s,t_0)\in\Delta$.
\end{itemize}
\begin{definition}
A map $P:\mathbb{R}_+\to\mathcal{B}(X)$ is called
	\begin{itemize}
		\item[(i)] \textit{a family of projectors} on $X$ if 
		$$P^2(t)=P(t),\text{ for every } t\geq 0;$$
		\item [(ii)] \textit{invariant} for the evolution operator $U:\Delta\to\mathcal{B}(X)$ if
		\begin{align*}
		U(t,s)P(s)x=P(t)U(t,s)x, 
		\end{align*}
		for all $(t,s,x)\in \Delta\times X$;
		\item[(iii)] \textit{stronlgy invariant} for the evolution operator $U:\Delta\to\mathcal{B}(X)$ if it is invariant for $U$ and for all $(t,s)\in\Delta$ the restriction of $U(t,s)$ on Range $P(s)$ is an isomorphism from Range $P(s)$ to Range $P(t)$.
	\end{itemize}
	\end{definition}
\begin{remark}
	It is obvious that if $P$ is strongly invariant for $U$ then it is also invariant for $U$. The converse is not valid (see \cite{mihit}).
\end{remark}
\begin{remark}\label{rem-proiectorstrong}
	If the family of projectors $P:\mathbb{R}_+\to\mathcal{B}(X)$ is strongly invariant for the evolution operator $U:\Delta\to\mathcal{B}(X)$ then (\cite{lupa}) there exists a map $V:\Delta\to\mathcal{B}(X)$ with the properties: 
	\begin{itemize}
		\item[$v_1)$ ] $V(t,s)$ is an isomorphism from Range $ P(t)$ to Range $ P(s)$, 
		\item [$v_2)$ ] $U(t,s)V(t,s)P(t)x=P(t)x$,
		\item[$v_3)$ ] $V(t,s)U(t,s)P(s)x=P(s)x$,
		\item[$v_4)$ ]$V(t,t_0)P(t)=V(s,t_0)V(t,s)P(t)$,
		\item[$v_5)$ ]$V(t,s)P(t)=P(s)V(t,s)P(t)$,
		\item[$v_6)$ ] $V(t,t)P(t)=P(t)V(t,t)P(t)=P(t)$,
		\end{itemize}
	for all $(t,s),(s,t_0)\in \Delta$ and $x\in X$.
\end{remark}
\begin{definition}
	Let $P_1,P_2,P_3:\mathbb{R}\to\mathcal{B}(X)$ be three families of projectors on $X$. We say that the family $\mathcal{P}=\{P_1,P_2,P_3\}$ is 
	\begin{itemize}
		\item [(i)] \textit{orthogonal} if
		\begin{itemize}
			\item [$o_1)$]$P_1(t)+P_2(t)+P_3(t)=I$ for every $t\geq 0$\\
			and
			\item[$o_2)$] $P_i(t)P_j(t)=0$ for all $t\geq 0$ and all $i,j\in\{1,2,3\}$ with $i\neq j$;
		\end{itemize}
	\item[(ii)] \textit{compatible} with the evolution operator $U:\Delta\to\mathcal{B}(X)$ if
	\begin{itemize}
		\item[$c_1)$] $P_1$ is invariant for $U$\\
		and
		\item[$c_2)$] $P_2,P_3$ are strongly invariant for $U$.
	\end{itemize}
	\end{itemize}
\end{definition}
In what follows we shall denote by $V_j(t,s)$ the isomorphism (given by Remark \ref{rem-proiectorstrong}) from Range $P_j(t)$ to Range $P_j(s)$ and $j\in\{2,3\}$, where $\mathcal{P}=\{P_2,P_2,P_3\}$ is compatible with $U.$
\begin{definition}
We say that a nondecreasing map $h:\mathbb{R}_+\to[1,\infty)$ is a \textit{ growth rate} if 
\begin{align*}
\lim\limits_{t\to\infty}h(t)=\infty.
\end{align*}
\end{definition}
As particular cases of growth rates we remark:
\begin{itemize}
	\item [$r_1)$ ] \textit{exponential rates}, i.e.
	$h(t)=e^{\alpha t}$ with $\alpha>0;$
	\item [$r_2)$ ]\textit{polynomial rates}, i.e.
	$h(t)=(t+1)^\alpha$ with $\alpha>0.$
\end{itemize}
Let $\mathcal{P}=\{P_1, P_2, P_3\}$ be an orthogonal family of projectors which is compatible with the evolution operator $U:\Delta\to\mathcal{B}(X)$ and $h,k,\mu,\nu:\mathbb{R}_+\to[1,\infty)$ be four growth rates.
\begin{definition}\label{def-tricho}
	We say that the pair $(U,\mathcal{P})$ is \textit{$(h,k,\mu,\nu)$-trichotomic} (and we denote $(h,k,\mu,\nu)-t$) if there exists a nondecreasing function $N:\mathbb{R}_+\to [1,\infty)$ such that
	\begin{itemize}
		\item [$(ht_1)$ ]$h(t)\|U(t,s)P_1(s)x\|\leq N(s)h(s)  \|P_1(s)x\|$
		\item [$(kt_1)$ ]$k(t)\|P_2(s)x\|\leq N(t) k(s)  \|U(t,s)P_2(s)x\|$
		\item [$(\mu t_1)$ ]$\mu(s)\|U(t,s)P_3(s)x\|\leq N(s) \mu(t) \|P_3(s)x\|$
		\item [$(\nu t_1)$ ]$\nu(s)\|P_3(s)x\|\leq N(t)\nu(t) \|U(t,s)P_3(s)x\|,$
	\end{itemize}
	for all	$(t,s,x)\in \Delta\times X.$
\end{definition}
In particular, if the function $N$ is constant then we obtain  the  \textit{uniform $(h,k,\mu,\nu)$-trichotomy} property, denoted by $u-(h,k,\mu,\nu)-t$.
\begin{remark}
	As important particular cases of $(h,k,\mu,\nu)$-trichotomy we have:
	\begin{itemize}
		\item[(i)] \textit{(nonuniform) exponential trichotomy} ($et$) and respectively \textit{uniform exponential trichotomy} ($uet$) when the rates $h,k,\mu,\nu$ are exponential rates;
		\item[(ii)]\textit{(nonuniform) polynomial trichotomy} ($pt$) and respectively \textit{uniform polynomial trichotomy} ($upt$) when the rates $h,k,\mu,\nu$ are polynomial rates;
		\item[(iii)]\textit{(nonuniform) $(h,k)-$dichotomy} ($(h,k)-d$) respectively {uniform $(h,k)-$dichotomy} ($u-(h,k)-d$) for $P_3=0$;
		\item[(iv)] \textit{(nonuniform) exponential dichotomy} ($ed$) and respectively \textit{uniform exponential dichotomy} ($ued$) when $P_3=0$ and the rates $h,k$ are exponential rates;
		\item[(v)]\textit{(nonuniform) polynomial dichotomy} (p.d.) and respectively \textit{uniform polynomial dichotomy} ($upd$) when $P_3=0$ and the rates $h,k$ are polynomial rates;
	\end{itemize}
\end{remark}
	 It is obvious that if the pair $(U,\mathcal{P})$ is $u-(h,k,\mu,\nu)-t$ then it is also $(h,k,\mu,\nu)-t$ In general, the reverse of this statement is not valid, phenomenon illustrated by
\begin{example}
Let  $U:\Delta\to\mathcal{B}(X)$ be the evolution operator defined by
\begin{align}
U(t,s)=\frac{u(s)}{u(t)}\left( \frac{h(s)}{h(t)}P_1(s)+ \frac{k(t)}{k(s)}P_2(s)+\frac{\mu(t)}{\mu(s)}\frac{\nu(s)}{\nu(t)}P_3(s)\right) 
\end{align}
where $u,h,k,\mu,\nu:\mathbb{R}_+\to[1,\infty)$ are growth rates and $P_1, P_2, P_3:\mathbb{R}_+\to\mathcal{B}(X)$ are projectors families on $X$ with the properties:
\begin{itemize}
	\item[(i)]$P_1(t)+P_2(t)+P_3(t)=I$ for every $t\geq 0$;
	\item[(ii)]\[ P_i(t)P_j(s)=\left\{ \begin{array}{ll} 0&\mbox{if $i\neq j$} \\
	P_i(s),& \mbox{ if $i=j$},
	\end{array} \right. \]for all $(t,s)\in\Delta.$
	\item[(iii)] $U(t,s)P_i(s)=P_i(t)U(t,s)$ for all $(t,s)\in\Delta$ and all $i\in\{1,2,3\}$.
\end{itemize}
For example if $P_1,P_2,P_3$ are constant and orthogonal then the conditions (i),(ii) and (iii) are satisfied.

 We observe that
\begin{align*}
h(t)\|U(t,s)P_1(s)x\|&=\frac{u(s)h(s)}{u(t)}\|P_1(s)x\|\leq u(s) h(s)\|P_1(s)x\|\\
 u(t)k(s)\|U(t,s)P_2(s)x\|&=u(s){k(s)}\|P_2(s)x\|\geq  k(t)\|P_2(s)x\|\\
 \mu(s)\|U(t,s)P_3(s)x\|&=\frac{u(s)\mu(t)\nu(s)}{u(t)\nu(t)}\|P_3(s)x\|\leq u(s)\mu(t)\|P_3(s)x\|\\
 u(t)\nu(t)\|U(t,s)P_3(s)x\|&=\frac{u(s)\nu(s)\mu(t)}{\mu(s)}\|P_3(s)x\|\geq  \nu(s)\|P_3(s)x\|
\end{align*} for all $(t,s,x)\in\Delta\times X.$

Thus the pair $(U,\mathcal{P})$ is  $(h,k,\mu,\nu)-t$.	\\
If we assume that the pair $(U,\mathcal{P})$ is $u-(h,k,\mu,\nu)-t$ then there exists a real constant $N\geq 1$ such that
\begin{align*}
N u(s)\geq u(t),\text{ for all } (t,s)\in\Delta. 
\end{align*}
Taking $s=0$ we obtain a contradiction.
\end{example}
\begin{remark}
The previous example shows that for all four growth rates $h,k,\mu,\nu$ there exits a pair $(U,\mathcal{P})$ which is $(h,k,\mu,\nu)-t$ and is not $u-(h,k,\mu,\nu)-t$.	
	\end{remark}
In the particular case when $\mathcal{P}$ is compatible with $U$ a characterization of $(h,k,\mu,\nu)-t$ is given by
\begin{proposition}\label{prop strong invariant trichotomy}
	If $\mathcal{P}=
	\{P_1,P_2,P_3\}$ is compatible with the evolution operator $U:\Delta\to\mathcal{B}(X)$ then the pair $(U,\mathcal{P})$ is $(h,k,\mu,\nu)$-trichotomic if and only if there exists a nondecreasing function $ N_1:\mathbb{R}_+\to[1,\infty)$ such that
	\begin{itemize}
		\item [$(ht_2)$ ]$h(t)\|U(t,s)P_1(s)x\|\leq N_1(s)h(s)\|x\|$
		\item [$(kt_2)$ ]$k(t)\|V_2(t,s)P_2(t)x\|\leq N_1(t) k(s)  \|x\|$
		\item [$(\mu t_2)$ ]$\mu(s)\|U(t,s)P_3(s)x\|\leq N_1(s) \mu(t)\|x\|$
		\item[$(\nu t_2)$ ]$\nu(s)\|V_3(t,s)P_3(t)x\|\leq N_1(t) \nu(t)  \|x\|$
	\end{itemize}
	for all $ (t,s,x)\in \Delta\times X$, where $V_j(t,s)$ for $j\in\{2,3\}$ is the isomorphism from Range $P_j(t)$ to Range $P_j(s)$.	
\end{proposition}
\begin{proof}
	\textit{Necessity.} By Remark \ref{rem-proiectorstrong} and the Definition \ref{def-tricho} we obtain
	\begin{align*}
(ht_2)\thickspace& \thickspace	h(t)\|U(t,s)P_1(s)x\|\leq N(s)h(s)\|P_1(s)x\|\leq N(s)\|P_1(s)\|h(s)\|x\|\\
&\leq N_1(s)h(s)\|x\|\\
	(kt_2)\thickspace &\thickspace k(t)\|V_2(t,s)P_2(t)x\|=k(t)\|P_2(s)V_2(t,s)P_2(t)x\|\\
	&\leq N(t)k(s)\|U(t,s)P_2(s)V_2(t,s)P_2(t)x\|\\
	&=N(t)k(s)\|P_2(t)x\|\leq N(t)\|P_2(t)\|k(s)\|x\|\leq N_1(t)k(s)\|x\|\\
	(\mu t_2)\thickspace &\thickspace \mu(s)\|U(t,s)P_3(s)x\|\leq N(s) \mu(t)\|P_3(s)x\|\leq N(s)\|P_3(s)\|\mu(t)\|x\|\\
	&\leq N_1(s)\mu(t)\|x\|\\
(\nu t_2 )\thickspace&\thickspace	\nu(s)\|V_3(t,s)P_3(t)x\|=\nu(s)\|P_3(s)V_3(t,s)P_3(t)x\|\\
&\leq N(t)\nu(t)\|U(t,s)P_3(s)V_3(t,s)P_3(t)x\|\\
	&=N(t)\nu(t)\|P_3(t)x\|\leq N(t)\|P_3(t)\|\nu(t)\|x\|\leq N_1(t)\nu(t)\|x\|,
	\end{align*}for all $(t,s,x)\in\Delta\times X,$ where 
	$$N_1(t)=\sup_{s\in[0,t]}N(s)(\|P_1(s)\|+\|P_2(s)\|+\|P_3(s)\|).$$
	
	\textit{Sufficiency.} The implications $(ht_2)\Rightarrow(ht_1)$ and $(\mu t_2)\Rightarrow(\mu t_1)$ result by replacing $x$ with $P_1(s)x$ respectively by $P_3(s)x$.
	
	For the implications $(kt_2)\Rightarrow(kt_1)$ and $(\nu t_2)\Rightarrow (\nu t_1)$ we have (by Remark \ref{rem-proiectorstrong})
	\begin{align*}
	k(t)\|P_2(s)x\|&=k(t)\|V_2(t,s)U(t,s)P_2(s)x\|\leq N(t)k(s)\|U(t,s)P_2(s)x\|\\
	&\text{and}\\
	\nu(s)\|P_3(s)x\|&=\nu(s)\|V_3(t,s)U(t,s)P_3(s)x\|\leq N(t)\nu(t)\|U(t,s)P_3(s)x\|,
	\end{align*}for all $(t,s,x)\in\Delta\times X.$
\end{proof}
A similar characterization for the $u-(h,k,\mu,\nu)-t$ concept results under the hypotheses of boundedness of the projectors $P_1,P_2,P_3$. A characterization with compatible family of projectors without assuming the boundedness of projectors is given by
\begin{proposition}\label{prop strong invariant trichotomy uniform}
	If $\mathcal{P}=
	\{P_1,P_2,P_3\}$ is compatible with the evolution operator $U:\Delta\to\mathcal{B}(X)$ then the pair $(U,\mathcal{P})$ is uniformly$-(h,k,\mu,\nu)-$ trichotomic if and only if there exists a constant $ N\geq 1$ such that
	\begin{itemize}
		\item [$(uht_1)$ ]$h(t)\|U(t,s)P_1(s)x\|\leq Nh(s)\|P_1(s)x\|$
		\item [$(ukt_1)$ ]$k(t)\|V_2(t,s)P_2(t)x\|\leq N k(s)  \|P_2(t)x\|$
		\item [$(u\mu t_1)$ ]$\mu(s)\|U(t,s)P_3(s)x\|\leq N \mu(t)\|P_3(s)x\|$
		\item[$(u\nu t_1)$ ]$\nu(s)\|V_3(t,s)P_3(t)x\|\leq N \nu(t)  \|P_3(t)x\|$
	\end{itemize}
	for all $ (t,s,x)\in \Delta\times X$, where $V_j(t,s)$ for $j\in\{2,3\}$ is the isomorphism from Range $P_j(t)$ to Range $P_j(s)$.	
\end{proposition}
\begin{proof}
	It is similar to the proof of Proposition \ref{prop strong invariant trichotomy}.
\end{proof}
\section{The main result}
In this section we give a characterization of $(h,k,\mu,\nu)-$trichotomy in terms of a certain type of uniform $(h,k,\mu,\nu)-$trichotomy using families of norms equivalent with the norms of $X$. Firstly we introduce
\begin{definition}\label{def-norma-compatibila}
	A family $\mathcal{N}=\{\|\cdot\|_t: t\geq0\}$ of norms on the Banach space $X$ (endowed with the norm $\|\cdot\|$) is called \textit{compatible} to the norm $\|\cdot\|$ if there exists a nondecreasing map $C:\mathbb{R}_+\to[1,\infty)$ such that 
	\begin{align}
	\|x\|&\leq \|x\|_t\leq C(t)\|x\|,\label{normprop-fara-proiectori1}
	\end{align}
		for all $(t,x)\in\mathbb{R}_+\times X$.
\end{definition}
\begin{proposition}\label{ex-norma-trichotomie2}
	If the pair $(U,\mathcal{P})$ is $(h,k,\mu,\nu)-t$ then the family of norms $\mathcal{N}_1=\{\|\cdot\|_t:t\geq0\}$ given by
	\begin{align}
	\|x\|_t&=\sup_{\tau\geq t} \frac{h(\tau)}{h(t)}\|U(\tau,t)P_1(t)x\|+\sup_{r\leq t} \frac{k(t)}{k(r)}\|V_2(t,r)P_2(t)x\|\nonumber\\
	&+\sup_{\tau\geq t} \frac{\mu(t)}{\mu(\tau)}\|U(\tau,t)P_3(t)x\|\label{norma-tricho-sus}
	\end{align}
	is compatible with $\|\cdot\|$.
\end{proposition}
\begin{proof}
	For $\tau=t=r$ in (\ref{norma-tricho-sus}) we  obtain that
	\begin{align*}
	\|x\|_t&\geq \|P_1(t)x\|+\|P_2(t)x\|+\|P_3(t)x\|\geq \|x\|
	\end{align*}
	for all $t\geq 0$.
	
	If the pair $(U,\mathcal{P})$ is $(h,k,\mu,\nu)-t$ then by Proposition  \ref{prop strong invariant trichotomy} there exits a nondecreasing function $N_1:\mathbb{R}_+\to\mathcal{B}(X)$ such that
	\begin{align*}
	\|x\|_t\leq 3N_1(t)\|x\|, \text{ for all } (t,x)\in\mathbb{R}_+\times X.
	\end{align*}
	Finally we obtain that $\mathcal{N}_1$ is compatible with $\|\cdot\|.$
\end{proof}
\begin{proposition}\label{ex-norma-trichotomie1}
	If the pair $(U,\mathcal{P})$ is $(h,k,\mu,\nu)-t$ then the family of norms $\mathcal{N}_2=\{\||\cdot\||_t,t\geq0\}$ defined by
	\begin{align}
	\||x\||_t&=\sup_{\tau\geq t} \frac{h(\tau)}{h(t)}\|U(\tau,t)P_1(t)x\|+ \sup_{r\leq t} \frac{k(t)}{k(r)}\|V_2(t,r)P_2(t)x\|\nonumber\\
	&+
	\sup_{r\leq t} \frac{\nu(r)}{\nu(t)}\|V_3(t,r)P_3(t)x\|\label{norma-tricho-jos}
	\end{align}
	is compatible with $\|\cdot\|.$
\end{proposition}
\begin{proof}
If the pair $(U,\mathcal{P})$ is $(h,k,\mu,\nu)-t$ then by Proposition  \ref{prop strong invariant trichotomy} there exits a nondecreasing function $N_1:\mathbb{R}_+\to\mathcal{B}(X)$ such that 
	\begin{align*}
	\||x\||_t\leq 3N_1(t)\|x\|, \text{ for all } (t,x)\in\mathbb{R}_+\times X.
	\end{align*}
	On the other hand, for  $\tau=t=r$ in the definition of $\||\cdot\||_t$ we  obtain
	\begin{align*}
	\||x\||_t&\geq \|P_1(t)x\|+\|P_2(t)x\|+\|P_3(t)x\|\geq\|x\|.
	\end{align*}
In consequence, by Definition \ref{def-norma-compatibila} it results that the family of norms $\mathcal{N}_2$ is compatible to $\|\cdot\|.$
\end{proof}	
	
The main result of this paper is
\begin{theorem}\label{unif=neunif-trichotomie}If $\mathcal{P}=\{P_1,P_2,P_3\}$ is compatible with the evolution operator $U:\Delta\to\mathcal{B}(X)$ then
	the pair $(U,\mathcal{P})$ is $(h,k,\mu,\nu)$-trichotomic if and only if there exist two families of norms $\mathcal{N}_1=\{\|\cdot\|_t: t\geq0\}$ and $\mathcal{N}_2=\{\||\cdot\||_t: t\geq0\}$ compatible with the norm $\|\cdot\|$ such that the following take place
	\begin{itemize}
		\item[($ht_3$) ] $h(t)\|U(t,s)P_1(s)x\|_t\leq  h(s) \|P_1(s)x\|_s$
		\item[($kt_3$) ] $k(t)\||V_2(t,s)P_2(t)x\||_s\leq  k(s) \||P_2(t)x\||_t$
		\item [$(\mu t_3)$ ]$\mu(s)\|U(t,s)P_3(s)x\|_t\leq  \mu(t) \|P_3(s)x\|_s$
		\item[$(\nu t_3)$ ]$\nu(s)\||V_3(t,s)P_3(t)x\||_s\leq  \nu(t)  \||P_3(t)x|\|_t$
	\end{itemize}
	 for all $(t,s,x)\in\Delta\times X.$
\end{theorem}
\begin{proof}
	\textit{Necessary.}
	If the pair $(U,\mathcal{P})$ is  $(h,k,\mu,\nu)$-trichotomic then by Propositions \ref{ex-norma-trichotomie2}  and  \ref{ex-norma-trichotomie1} that there exist the families of norms $\mathcal{N}_1=\{\|\cdot\|_t: t\geq0\}$ and $\mathcal{N}_2=\{\||\cdot\||_t: t\geq0\}$ compatible with $\|\cdot\|$.

$\boldsymbol{(ht_1)\Rightarrow(ht_3)}.$ We have that
\begin{align*}
h(t)\|U(t,s)P_1(s)x\|_t&=h(t)\|P_1(t)U(t,s)P_1(s)x\|_t\\
&=h(t)\sup_{\tau\geq t} \frac{h(\tau)}{h(t)}\|U(\tau,t)P_1(t)U(t,s)P_1(s)x\|\\
&\leq h(s)\sup_{\tau\geq s} \frac{h(\tau)}{h(s)}\|U(\tau,s)P_1(s)x\|= h(s)\|P_1(s)\|_s,
\end{align*}for all $(t,s,x)\in\Delta\times X$.

$\boldsymbol{(kt_2)\Rightarrow(kt_3)}.$ If $(kt_2)$ holds then
\begin{align*}
k(t)\||V_2(t,s)P_2(t)x\||_s&=k(t)\||P_2(s)V_2(t,s)P_2(t)x\||_s\\
&=k(t)\sup_{r\leq s} \frac{k(s)}{k(r)}\|V_2(s,r)P_2(s)V_2(t,s)P_2(t)x\|\\
&\leq k(s) \sup_{r\leq t} \frac{k(t)}{k(r)}\|V_2(t,r)P_2(t)x\|= k(s)\||P_2(t)\||_t
\end{align*}for all $(t,s,x)\in\Delta\times X$.

$\boldsymbol{(\mu t_1)\Rightarrow(\mu t_3)}.$ If $(U,\mathcal{P})$ is $(h,k,\mu,\nu)-$ trichotomic then by $(\mu t_1)$ it results 
\begin{align*}
\mu(s)\|U(t,s)P_3(s)x\|_t&=\mu(s)\|P_3(t)U(t,s)P_3(s)x\|_t\\
&=\mu(s)\sup_{\tau\geq t} \frac{\mu(t)}{\mu(\tau)}\|U(\tau,t)P_3(t)U(t,s)P_3(s)x\|\\
&= \mu(s)\sup_{\tau\geq t}\frac{\mu(t)}{\mu(\tau)}\|U(\tau,s)P_3(s)x\| \leq\mu(t)\sup_{\tau\geq s}\frac{\mu(s)}{\mu(\tau)}\|U(\tau,s)P_3(s)x\|\\
&=\mu(t)\|P_3(s)x\|_s,
\end{align*}for all $(t,s,x)\in\Delta\times X$.

$\boldsymbol{(\nu t_2)\Rightarrow(\nu t_3)}.$ Using Proposition \ref{normprop-fara-proiectori1} we obtain
\begin{align*}
\nu(s)\||V_3(t,s)P_3(t)x\||_s&=\nu(s)\||P_3(s)V_3(t,s)P_3(t)x\||_s\\
&=\nu(s)\sup_{r\leq s} \frac{\nu(r)}{\nu(s)}\|V_3(s,r)P_3(s)V_3(t,s)P_3(t)x\|\\
&\leq \nu(t)\sup_{r\leq t}\frac{\nu(r)}{\nu(t)}\|V_3(t,r)P_3(t)x\| =\nu(t)\||P_3(t)x\||_t,
\end{align*}
for all $(t,s,x)\in\Delta\times X$.

\textit{Sufficiency.}We assume that there are two families of norms $\mathcal{N}_1=\{\|\cdot\|_t: t\geq0\}$ and $\mathcal{N}_2=\{\||\cdot\||_t: t\geq0\}$ compatible with the norm $\|\cdot\|$ such that the inequalities $(ht_3)--(\nu t_3)$ take place. Let $(t,s,x)\in\Delta\times X$,

$\boldsymbol{(ht_3)\Rightarrow(ht_2)}.$ The inequality $(ht_3)$ and Definition \ref{def-norma-compatibila} imply that
\begin{align*}
h(t)\|U(t,s)P_1(s)x\|&\leq \|U(t,s)P_1(s)x\|_t\leq  h(s)\|P_1(s)x\|_s\\
&\leq h(s)C(s)\|P_1(s)x\| \leq C(s)\|P_1(s)\|h(s) \|x\|.
\end{align*}
$\boldsymbol{(kt_3)\Rightarrow(kt_2)}.$ Similarly,
\begin{align*}
k(t)\|V_2(t,s)P_2(t)x\|&\leq k(t)\||V_2(t,s)P_2(t)x\||_s\leq  k(s)\||P_2(t)\||_t\\
&\leq  k(s)C(t)\|P_2(t)x\|\leq C(t)\|P_2(t)\|k(s)\|x\|. 
\end{align*}
$\boldsymbol{(\mu t_3)\Rightarrow(\mu t_2)}.$ From Definition \ref{def-norma-compatibila} and inequality $(\mu t_3)$ we have
\begin{align*}
\mu(s)\|U(t,s)P_3(s)x\|&\leq \mu(s) \|U(t,s)P_3(s)x\|_t
\leq \mu(t)\|P_3(s)x\|_s\\
&\leq C(s) \mu(t)\|P_3(s)x\|\leq C(s)\|P_3(s)\| \mu(t)\|x\|.
\end{align*}
$\boldsymbol{(\nu t_3)\Rightarrow(\nu t_2)}.$ Similarly,
\begin{align*}
\nu(s)\|V_3(t,s)P_3(t)x\|&|\leq \nu(s)\||V_3(t,s)P_3(s)x\||_s\leq \nu(t)\||P_3(t)x\||_t\\
&\leq C(t) \nu(t)\|P_3(t)x\|\leq  C(t)\|P_3(t)\| \nu(t)\|x\| .
\end{align*}
If we denote by
$$N(t)=\sup_{s\in[0,t]}C(s)(\|P_1(s)\|+\|P_2(s)\|+\|P_3(s)\|)$$
then we obtain that the inequalities $(ht_2),(kt_2),(\mu t_2),(\nu t_2)$ are satisfied. By Proposition \ref{prop strong invariant trichotomy} it follows that $(U,\mathcal{P})$ is $(h,k,\mu,\nu)-t$.
\end{proof}
As a particular case, we obtain a characterization of (nonuniform) exponential trichotomy given by
\begin{corollary}\label{cor1}
	If $\mathcal{P}=\{P_1,P_2,P_3\}$ is compatible with the evolution operator $U:\Delta\to\mathcal{B}(X)$ then
	the pair $(U,\mathcal{P})$ is exponential trichotomic if and only if there are four real constants $\alpha,\beta,\gamma,\delta>0$ and two families of norms $\mathcal{N}_1=\{\|\cdot\|_t: t\geq0\}$ and $\mathcal{N}_2=\{\||\cdot\||_t: t\geq0\}$ compatible with the norm $\|\cdot\|$ such that
\begin{itemize}
	\item[($et_1$) ] $\|U(t,s)P_1(s)x\|_t\leq e^{-\alpha(t-s)}\|P_1(s)x\|_s$
	\item[($et_2$) ] $\||V_2(t,s)P_2(t)x\||_s\leq e^{-\beta(t-s)}\||P_2(t)x\||_t$
	\item [($e t_3$) ]$\|U(t,s)P_3(s)x\|_t\leq e^{\gamma(t-s)} \|P_3(s)x\|_s$
	\item[($e t_4$) ]$\||V_3(t,s)P_3(t)x\||_s\leq e^{\delta(t-s)}\||P_3(t)x\||_t$,
\end{itemize}
for all $(t,s,x)\in\Delta\times X.$
\end{corollary}
\begin{proof}
	It results from Theorem \ref{unif=neunif-trichotomie} for $$h(t)=e^{\alpha t},k(t)=e^{\beta t},\nu(t)=e^{\gamma t},\nu(t)=e^{\delta t},$$
with $\alpha,\beta,\gamma,\delta>0.$
\end{proof}
If the growth rates are of polynomial type then we obtain a characterization of (nonuniform) polynomial trichotomy given by
\begin{corollary}\label{cor2}
	Let  $\mathcal{P}=\{P_1,P_2,P_3\}$ is compatible with the evolution operator $U:\Delta\to\mathcal{B}(X)$.
	Then $(U,\mathcal{P})$ is nonuniform polynomial trichotomic if and only if there exist two families of norms $\mathcal{N}_1=\{\|\cdot\|_t: t\geq0\}$ and $\mathcal{N}_2=\{\||\cdot\||_t: t\geq0\}$ compatible with the norm $\|\cdot\|$ and four real constants $\alpha,\beta,\gamma,\delta>0$ such that
	\begin{itemize}
		\item[($pt_1$) ] $(t+1)^\alpha\|U(t,s)P_1(s)x\|_t\leq (s+1)^\alpha\|P_1(s)x\|_s$
		\item[($pt_2$) ] $(t+1)^\beta\||V_2(t,s)P_2(t)x\||_s\leq (s+1)^\beta\||P_2(t)x\||_t$
		\item [($p t_3$) ]$(s+1)^\gamma\|U(t,s)P_3(s)x\|_t\leq (t+1)^\gamma \|P_3(s)x\|_s$
		\item[($p t_4$) ]$(s+1)^\delta\||V_3(t,s)P_3(t)x\||_s\leq (t+1)^\delta
		\||P_3(t)x\||_t$,
	\end{itemize}
	for all $(t,s,x)\in\Delta\times X.$
\end{corollary}
\begin{proof}
	It results from Theorem \ref{unif=neunif-trichotomie} for $$h(t)=(t+1)^\alpha,k(t)=(t+1)^\beta,\mu(t)=(t+1)^\gamma,\nu(t)=(t+1)^\delta,$$
	with $\alpha,\beta,\gamma,\delta>0.$
\end{proof}
\begin{definition}
	A family of norms $\mathcal{N}=\{\|\cdot\|_t,t\geq 0\}$ is \textit{uniformly compatible} with the norm $\|\cdot\|$ if there exits a constant $c>0$
  such that
\begin{equation}\label{norma compatibila uniform}
\|x\|\leq \|x\|_t\leq c\|x\|, \text{ for all } (t,x)\in\mathbb{R}_+\times X.
\end{equation}
\end{definition}
\begin{remark}
	From the proofs of Propositions \ref{ex-norma-trichotomie2}, \ref{ex-norma-trichotomie1} it results that if the pair $(U,\mathcal{P})$ is uniformly $(h,k,\mu,\nu)-$ trichotomic then the families of norms $\mathcal{N}_1=\{\|\cdot\|_t:t\geq 0\}$ and $\mathcal{N}_2=\{\||\cdot\||_t:t\geq 0\}$ (given by (\ref{norma-tricho-sus}) and (\ref{norma-tricho-jos})) are uniformly compatible with the norm $\|\cdot\|.$
\end{remark}
A characterization of the uniform$-(h,k,\mu,\nu)-$trichotomy is given by
\begin{theorem}
	Let $\mathcal{P}=\{P_1,P_2,P_3\}$ be compatible with the evolution operator $U:\Delta\to\mathcal{B}(X)$. Then
	the pair $(U,\mathcal{P})$ is uniformly$-(h,k,\mu,\nu)-$trichotomic if and only if there exist two families of norms $\mathcal{N}_1=\{\|\cdot\|_t: t\geq0\}$ and $\mathcal{N}_2=\{\||\cdot\||_t: t\geq0\}$ uniformly compatible with the norm $\|\cdot\|$ such that the inequalities $(ht_3),(kt_3),(\mu t_3)$ and $(\nu t_3)$ are satisfied.
\end{theorem}
\begin{proof}
It results from the proof of Theorem \ref{unif=neunif-trichotomie} (via Proposition \ref{prop strong invariant trichotomy uniform}).
\end{proof}
\begin{remark}
	Similarly as in Corollaries \ref{cor1}, \ref{cor2} one can obtain characterizations for uniform exponential trichotomy respectively uniform polynomial trichotomy.
\end{remark}
Another characterization of the $(h,k,\mu,\nu)-$trichotomy is given by
\begin{theorem}\label{unif=neunif-trichotomie-fara-proiectori}If $\mathcal{P}=\{P_1,P_2,P_3\}$ is compatible with the evolution operator $U:\Delta\to\mathcal{B}(X)$ then
	the pair $(U,\mathcal{P})$ is $(h,k,\mu,\nu)$-trichotomic if and only if there exist two families of norms $\mathcal{N}_1=\{\|\cdot\|_t,t\geq 0\},\mathcal{N}_2=\{\||\cdot\||_t: t\geq0\}$ compatible with the family of projectors $\mathcal{P}=\{P_1,P_2,P_3\}$ such that
	\begin{itemize}
		\item[($ht_4$) ] $h(t)\|U(t,s)P_1(s)x\|_t\leq  h(s) \|x\|_s$
		\item[($kt_4$) ] $k(t)\||V_2(t,s)P_2(t)x\||_s\leq  k(s) \||x\||_t$
		\item [$(\mu t_4)$ ]$\mu(s)\|U(t,s)P_3(s)x\|_t\leq  \mu(t) \|x\|_s$
		\item[$(\nu t_4$) ]$\nu(s)\||V_3(t,s)P_3(t)x\||_s\leq  \nu(t)  \||x|\||_t$
	\end{itemize}
	for all $(t,s,x)\in\Delta\times X.$
\end{theorem}
\begin{proof}
	\textit{Necessity.} It results from Theorem \ref{unif=neunif-trichotomie} and inequalities
	\begin{eqnarray*}
	\|P_i(t)x\|_t\leq \|x\|_t&\text{and}&
	\||P_i(t)x\||_t\leq \||x\||_t,
	\end{eqnarray*}for all $(t,x)\in\mathbb{R}_+\times X$ and $i=\{1,2,3\}$.
	
	\textit{Sufficiency.} It results replacing $x$ by $P_1(s)x$ in $(ht_4)$, $x$ by  $P_2(t)x$ in $(kt_4)$, $x$ by $P_3(s)x$ in $(\mu t_4)$ and $x$ by $P_3(t)x$ in $(\nu t_4)$.
\end{proof}
The variant of the previous theorem for uniform $(h,k,\mu,\nu)-$trichotomy is given by
\begin{theorem}
	If $\mathcal{P}=\{P_1,P_2,P_3\}$ is compatible with the evolution operator $U:\Delta\to\mathcal{B}(X)$ then
	the pair $(U,\mathcal{P})$ is uniformly $-(h,k,\mu,\nu)-$ trichotomic if and only if there exist two families of norms $\mathcal{N}_1=\{\|\cdot\|_t:t\geq 0\},\mathcal{N}_2=\{\||\cdot\||_t: t\geq0\}$ uniformly compatible with the family of projectors $\mathcal{P}=\{P_1,P_2,P_3\}$ such that the inequalities $(ht_4),(kt_4),(\mu t_4 )$ and $(\nu t_4)$ are satisfied.
\end{theorem}
\begin{proof}
	It is similar with the proof of Theorem \ref{unif=neunif-trichotomie}.
\end{proof}
\begin{remark}
	If the growth rates are exponential respectively polynomial then we obtain characterizations for exponential trichotomy, uniform exponential trichotomy and uniform polynomial trichotomy.
\end{remark}

\end{document}